\documentclass{amsart}

\usepackage{amssymb}
\usepackage{amsmath}
\usepackage{amscd}
\usepackage{amsthm}

\newcommand{\C} {\mathbb C}

\newcommand{\fd}{\textit}    

\newtheorem{theorem}{Theorem}[section]
\newtheorem{lemma}[theorem]{Lemma}
\newtheorem{corollary}[theorem]{Corollary}
\newtheorem{proposition}[theorem]{Proposition}

\theoremstyle{definition}

\newtheorem{definition}[theorem]{Definition}

\newtheorem*{thrichotomy}{Thrichotomy}

\begin{document}

\title{On N\"orlund-Voronoi summability and instability of rational maps}

\author{Carlos Cabrera \textsuperscript{*} }

\address{Carlos Cabrera: Instituto de Matem\'aticas\\ Unidad Cuernavaca. UNAM}

\email{carloscabrerao@im.unam.mx}

\author{Peter Makienko}
\address{Peter Makienko: Instituto de Matem\'aticas\\ Unidad Cuernavaca. UNAM}
\email{makienko@im.unam.mx}
\author{Alfredo Poirier }

\address{Alfredo Poirier: Departamento de Ciencias, Secci\'on Matem\'aticas. PUCP}
\email{apoirie@pucp.edu.pe}
\dedicatory{In honor of Misha Lyubich's  60th 
Birthday.}

\begin{abstract}
We investigate the connection between the instability of rational maps and 
summability methods applied to the spectrum of a critical point belonging to the Julia set of
a rational map.  
\end{abstract}
\maketitle

\footnotetext{This work was partially supported by CONACYT CB2015/255633 and PAPIIT IN102515.}

\section{Motivation and main results}

Let $Rat_d$ be the space of all rational maps $R$ of degree $d>0$ defined 
on the Riemann  sphere $\bar{\C}$. Let $Crit(R)$ be the set of critical points of $R$. For $c\in Crit(R)$ the  \textit{individual postcritical} set is 
$$P_c(R)=\overline{ \bigcup_{n>0} R^n(c)}.$$ The \textit{postcritical set} of $R$ is
$$P(R)=\overline{\bigcup_{c\in Crit(R)} P_c(R)}.$$  
The \textit{Julia set} $J(R)$ is the accumulation set of all repelling periodic cycles.
A rational map $R$ is called \textit{hyperbolic} if the postcritical set does not intersect 
the Julia set. 

The Fatou conjecture states that hyperbolic rational maps of degree $d$ form 
an open and dense subset of $Rat_d$. A rational map 
$R\in Rat_d$ is called \textit{structurally stable}  if there exists a neighborhood $U$ of 
$R$ in $Rat_d$ such that for every $Q$ in $U$  there is a quasiconformal conjugation 
between $R$ and $Q$.  A theorem of R. Ma\~n\'e, P. Sad, and D. Sullivan shows that  the set of 
structurally stable maps forms an open and dense subset of $Rat_d$ (see \cite{MSS}). 
Hence the Fatou conjecture can be reformulated in the following way.

\textit{If $R\in Rat_d$ has a critical point $c$ in the Julia set $J(R)$, 
then $R$ is not structurally stable or, 
equivalently, is an unstable map.}  

Recall that a critical point $c$ of a rational map $R$ is called \textit{summable} if the series $\displaystyle{\sum_{n= 0}^\infty \frac{1}{(R^n)'(R(c))}}$ is absolutely convergent.

 Among other results, A. Avila \cite{AvilaInfPert}, G. Levin \cite{LevinAnalytic}, and P. Makienko \cite{MakRuelle} proved the following statement. 

\begin{theorem}\label{th.AvLevMak} If $c\in J(R)$ is a summable critical point with non-zero sum then $R$ is not $J$-stable, whenever $P_c(R)\neq \overline{\C}$.  
\end{theorem}

The condition  $\sum_{n= 0}^\infty\frac{1}{(R^n)'(R(c))}=0$  is a special case which requires additional considerations 
(see again \cite{AvilaInfPert}, \cite{LevinAnalytic},  \cite{LevinPert}, \cite{MakRuelle} and compare with the discussion in Section 3.1).  However, in \cite{LevinPert} it was shown that, with additional conditions on the postcritical set, among all critical points in $J(R)$ with absolutely convergent series   $\sum_{n= 0}^\infty\frac{1}{(R^n)'(R(c))}$ there is at least one critical point with  $\sum_{n= 0}^\infty \frac{1}{(R^n)'(R(c))}\neq 0$.

In this paper we go along with the approach initiated by Avila, Levin, and Makienko to the cases where the series $\sum_{n\geq 0} \frac{1}{(R^n)'(R(c))}$ is absolutely divergent but has radius of 
convergence at least $1$ with respect to the N\"orlund-Voronoi summability method (see definitions and discussion below).  Also,
we present a class of measures  defining directions in the tangent space of $Rat_d$ at $R$ which are not generated by a quasiconformal deformation of $R$.

Given a point $a$, we call the sequence $$\sigma(a)=\{\sigma_n(a)\}=\left\{\frac{1}
{(R^n)'(R(a))}\right\}_{n=0}^\infty$$  \textit{the spectrum of $a$}. 

To avoid technicalities, in this article we always assume that
\begin{itemize}
\item \textit{the rational map $R$  fixes $0$,$1$ and $\infty$,}  \item \textit{there are no critical relations},  and

\item \textit{the individual postcritical sets under discussion are bounded}.
\end{itemize}

For example if there are no critical relations and  $P_c(R)$ does not contains all fixed points of $R$, then a suitable rotation of $R$ satisfies our assumption. However, in many cases our arguments do not need all of these restrictions.

Let $c$ be a critical point for $R$. 
Let us consider the following  trichotomy:
\begin{thrichotomy}\ 

\begin{enumerate}
 \item The sequence $\sigma_n(c)$ converges to $\infty$.
 \item There is a subsequence $\{n_i\}$ such that $\sigma_{n_i}(c)$ converges to $0$.
 \item The number $\liminf_n |\sigma_n(c)|$ is neither $ 0$ nor $ \infty.$ 
\end{enumerate}
\end{thrichotomy}

Let us comment on the trichotomy above. 

\begin{itemize}
\item Conjecturally, the case (1) of the thrichotomy should imply that the critical point $c$ belongs to the 
Fatou set. For example, let  $R$ be an unimodal polynomial whose critical point $c$ has spectrum $\sigma(c)$ with bounded 
multiplicative oscillation   ($\sup\left|\frac{\sigma_{n+1}(c)}{\sigma_n(c)}\right|<\infty$). Then $c$ belongs to the Fatou set by
Ma\~n\'e's  theorem (see \cite{Man}). 
 \item Theorem \ref{th.AvLevMak} deals with the
 case (2) for $\sigma(c)\in \ell_1$ (see \cite{AvilaInfPert}, \cite{LevinAnalytic},\cite{LevinPert} and  \cite{MakRuelle}). 
\item In case (3)  assume $\sigma(c)\in \ell_\infty$, then  $\sigma(c)$ has 
bounded multiplicative oscillation. If, for instance $R$ is unimodal, again by Ma\~n\'e's theorem 
the map $R$ is parabolic and hence is not structurally stable.
\end{itemize}

Where $\ell_1$ and $\ell_\infty$ are, as usual, the Banach spaces of absolutely summable and bounded sequences, respectively. 

The following proposition and its subsequent corollaries are the main motivations to consider the thrichotomy above.

\begin{proposition}\label{prop.motivation} Let $R$ be a stable rational map in the  complex one-parameter family $R_\lambda=R+\lambda$
and $c$ be a  critical  point on the Julia set  $J(R)$ with bounded individual 
postcritical set $P_c(R)$.  Consider the partial sums  $\displaystyle{S_n=\sum_{i=0}^{n}} \sigma_i(c)$.  Then for all $n$ we have $$\left|{S_n}\right|\leq C |\sigma_n(c)|,$$
for some constant $C.$
\end{proposition}

The proof is contained in Lemma 5 in \cite{MakRuelle} (see also Avila 
\cite{AvilaInfPert}). The quadratic polynomial case was also noted by  Levin 
in \cite{LevinAnalytic}. However, the proof is elementary 
and we include it in the next section for the sake of completeness.

\begin{corollary}\label{cor.motivation} Let $R$ be a  rational map 
and $c$ be a  critical  point on the Julia set  $J(R)$ with bounded $P_c(R)$.
Each of the following criteria implies that the map $R$ is unstable in the family $R_\lambda=R+\lambda$, for $\lambda\in \C$.

\begin{itemize}
 
\item In case (2) of the thrichotomy, a subsequence  $\sigma_{n_i}(c)$ converges to 
$0$ and satisfies $\limsup|S_{n_i}|>0$. 
 
\item In case (3) of the thrichotomy,  a subsequence  $\sigma_{n_i}(c)$ converges to a 
non-zero finite limit but  $\limsup |S_{n_i}|=\infty$. 
 \end{itemize}
\end{corollary}

\begin{corollary}\label{cor.positiverat}
 Let $R$ be a rational map such that  $R'\geq 0$ on $P_c(R)$ for a  critical point $c\in J(R)$ with $\sigma(c)$ satisfying either one of the conditions (2) or (3) in the trichotomy, then $R$ is unstable in the  family $R_\lambda=R+\lambda$, for $\lambda\in \C$.
\end{corollary}

The following theorem generalizes Corollary \ref{cor.positiverat}.

\begin{theorem}\label{th.Lyapunovnon}
 Let $R$ be a rational map. Assume there exists a point $x\in J(R)$ such that $R'(z)\geq 0$ for  $z\in \bigcup_{n=1}^\infty \{R^n(x)\}$ and satisfying the following conditions 
 \begin{enumerate}
  \item The set  $\bigcup_{n=1}^\infty \{R^n(x)\}$ is bounded
  and does not intersect the set of critical points of $R$.
  \item $R(x)$ has non-negative lower Lyapunov exponent. 
  \item $\sigma(x)\notin \ell_1$.
  
\end{enumerate}
Then $R$ is unstable in the family $R_\lambda=R+\lambda$, for $\lambda\in \C.$

Even more, if $x$ is a critical point then condition (3) is redundant. 
\end{theorem}

This theorem will be proved after 
Theorem \ref{TheoremA}.  Theorem \ref{th.Lyapunovnon} leads to the following question. 

\textit{Does there exists a rational map $R$ such that $\sigma(x)\in \ell_1$ for Lebesgue almost every point of $J(R)$?}

It seems that the non-negative condition of $R'$ on $P_c(R)$ is extremal and rare within rational dynamics. However, there are many examples of critical circle maps conjugated to  rational maps satisfying this condition. Here a critical circle map is a map which leaves the circle invariant with a critical point in the circle. For entire or meromorphic maps there are simple maps holding this condition. For example, consider  $\lambda e^{z^2+c}$, where $\lambda$ and $c$ are real numbers with $\lambda>0$, which have non-negative derivative on the orbit of $0$. 

In the next section, we will give a family of examples of rational maps with real coefficients and non-negative derivative on an individual postcritical set $P_c(R)$.  Thus Corollary \ref{cor.positiverat} produces new  examples of unstable real rational maps, compare with the class of real rational maps discussed by  W. Shen in \cite{Shenrat}.

The following corollary is the main motivation for the constructions in the present work.  To study  the spectrum $\sigma(c)$ of a critical point $c$, we will consider the Abel averages $A_\lambda=(1-\lambda) \sum_{n= 0}^\infty \sigma_n(c) \lambda^n$. Recall that a sequence $\{a_n\}$ is \textit{Abel convergent}
to $L$, which is allowed to be infinity, whenever $A_\lambda=(1-\lambda) \sum_{n= 0}^\infty a_n \lambda^n$ is finite for  $\lambda<1$ and  $$\lim_{\lambda\rightarrow 1} A_\lambda=L.$$

\begin{corollary}\label{Cor.AbelAver} Let $R$ be a rational map and $c$ be a critical point such that 
$\sigma_{n}(c)=o(n)$. Suppose further that the individual postcritical set $P_c(R)$ is 
bounded. If $\sigma(c)$ is not Abel convergent to $0$, then 
$R$ is an unstable map  in the one-parameter family $R_\lambda=R+\lambda$. 
\end{corollary}

Let us note that if a sequence $\{a_n\}$ converges then the sequence of Abel averages converges to the same limit.  The reciprocal fails to be 
true.
Many unbounded sequences are Abel convergent.

According to results by  H. Bruin and S. van Strien \cite{BruinStrienSumm} and, independently, by J. Rivera-Letelier \cite{RiveraSumma},  the Julia set $J(R)$ has Lebesgue measure $0$ whenever $J(R)$
contains only critical points $c$ satisfying
$\sigma(c)\in\ell_1$. Besides, examples given by X. Buff and A. Cheritat \cite{BuffCheritat} and  by A. Avila and M. Lyubich \cite{AviLyuPosMea} of 
quadratic polynomials with Julia set of positive measure, show that there are rational maps $R$ with a critical point $c\in J(R)$ with  $\sigma(c)\notin \ell_1.$  Perhaps $\sigma(c)$ is not in $\ell_1$ for every infinitely renormalizable unimodal map. The referee kindly pointed out that $\sigma(c)\notin \ell_1$ for every Feigenbaum quadratic polynomial.

If $J(R)$ contains several critical points and one of them, say $c$, has $\sigma(c)\in \ell_1 $, then $R$ is unstable whenever $P_c(R)$ satisfies additional restrictions (see  Avila 
\cite{AvilaInfPert},  Levin \cite{LevinAnalytic}, \cite{LevinPert}, and Makienko \cite{MakRuelle}). 

To formulate our main results we need the following construction.\\

\textbf{Voronoi measures}. First we discuss some weak conditions on the behavior of $\sigma(c)$, for $c\in J(R)$, that include many instances of  the cases (1), (2) and (3) of the thrichotomy above. In order to do so, we  consider  averaged sequences of the spectrum, that is, the N\"orlund-Voronoi averages. The construction of such averages uses the so-called N\"orlund summability method 
which was  firstly published  by G. Voronoi in 1902. An english version of Voronoi's  work can be found at \cite{VoronoiSums}.

Fix a sequence of non-negative real numbers $q_n\geq 0$ such that $q_0>0$ and $\lim_n\frac{q_n}{Q_n}=0$, where $Q_n=q_0+...
+q_n$ are the partial sums. For a sequence of complex numbers $\{x_n\}$, the values  
$$t_n=\frac{q_nx_0+q_{n-1}x_1+...
+q_0x_n}{Q_n}$$ are called the \textit{N\"orlund averages} with respect to the  sequence $\{q_n\}.$

If  $$\limsup \sqrt[n]{|t_n|}\leq 1,$$ then we say that $\{x_n\}$ 
is \textit{N\"orlund regular}, or \textit{$N$-regular} for short. 

Indeed, given a sequence $\{a_n\}$, the number $\liminf \frac{1}{\sqrt[n]{|a_n|}}$ is the radius of convergence of the power series $\sum_{n= 0}^\infty a_n z^n$. In what follows we will call the number $\liminf \frac{1}{\sqrt[n]{|a_n|}}$ the \textit{radius of convergence of the sequence} $\{a_n\}$.

A convenient way to think about the N\"orlund method is to regard the sequence $\{q_n\}$ as a linear operator $N:\Theta \rightarrow \Theta$, where  
$\Theta$ is the linear space of all complex sequences, and  $N$ is 
the infinite matrix with coordinates $\displaystyle{N_{m,n}=\frac{q_{m-n}}
{Q_m}}$ for $n\leq m$ and $N_{m,n}=0$ for $n>m.$ We call $N$ the \textit{N\"orlund matrix associated to} $\{q_n\}.$  Hence, a sequence $\{x_n\}$ is $N$-regular if and only if $N(\{x_n\})$  has radius of convergence at least $1.$

It is  known (see for example \cite{Boos}) that $N$ 
defines a continuous linear endomorphism of $\ell_\infty$. If $\mathcal{C}\subset \Theta$ is the subspace of all convergent sequences, then  $N(\mathcal{C})\subset \mathcal{C}$ and, furthermore, 
$\lim_n x_n=\lim_n t_n$, 
for $\{x_n\}\in \mathcal{C}$.  

The next lemma is not difficult to prove and appears as Lemma 3.3.10 in  \cite{Boos}.

\begin{lemma}\label{lemma.Norlund}
 Let $N$ be the N\"orlund matrix associated to $\{q_n\}$  and $\{x_n\}$ be a N\"orlund regular sequence. Then the following properties 
 hold.
 \begin{enumerate}
  \item The series $q(\lambda)=\sum_{n= 0}^\infty q_n \lambda^n$ and $Q(\lambda)=\sum_{n= 0}^\infty Q_n\lambda^n$ 
  converge for $|\lambda|<1.$
  \item The series $x(\lambda)=\sum_{n= 0}^\infty x_n\lambda^n$ converges in a 
  neighborhood of $0.$
  \item The convolution series $$[C_N(\{x_n\})](\lambda)=\sum_{n= 0}^\infty(\sum_{i=0}^n q_ix_{n-i})\lambda^n=\sum_{n= 0}^\infty t_n Q_n\lambda^n$$ converges for $|\lambda|<1.$ 
  
 \end{enumerate}

\end{lemma}

According to the lemma above, if $\{x_n\}$ is $N$-regular, then from the relations 
$$[C_N(\{x_n\})](\lambda)=x(\lambda)q(\lambda)$$ and $q_0\neq 0$ we conclude two facts. First,  that the radius of convergence of $\{x_n\}$ is 
non-zero and $x(\lambda)$ can be continued to a meromorphic function $X(\lambda)$ 
on the unit disk. Second, as $q(\lambda)$ has non negative Taylor coefficients, there are no zeros on the interval $[0,1)$, and so the  poles of $X(\lambda)$ lay outside $[0,1)$.

Reciprocally, if the sequence $\{x_n\}$ has non-zero radius of convergence and  
$x(\lambda)=\sum_{n= 0}^\infty x_n \lambda^n$ can be extended to a meromorphic function 
$X(\lambda)=\frac{\psi(\lambda)}{\phi(\lambda)}$ such that  $\phi$ and $\psi$ are
holomorphic and $\phi$ has
non-negative Taylor coefficients with $\phi(0)>0$. Then the Taylor coefficients of $\phi$ 
define a N\"orlund operator $N$ such that $\{x_n\}$ is an $N$-regular sequence.  For instance, if $\{x_n\}$ itself has radius of convergence at least 
$1$, then the choice $q_0=1$ and $q_n=0$ 
for all $n>0$ gives the identity matrix as a  N\"orlund matrix
or, equivalently,   $t_n=x_n$.

Another example of a  N\"orlund matrix is when $q_n=1$ for every $n$.  Then the 
N\"orlund averages of $\{x_n\}$ becomes the Ces\`aro averages.  Indeed, the N\"orlund averages is a 
generalization of  both iterated Ces\`aro and Abel averages for suitable sequences $\{q_n\}$. Moreover, Ces\`aro and Abel convergent sequences have  
non-zero radius of convergence.  In our situation, the spectrum $\sigma(c)$ 
has non-zero radius of convergence if and only if the lower Lyapunov exponent of $R(c)
$ is bounded below.  In fact, we believe the following conjecture holds true.
\smallskip

\textbf{Conjecture.} \textit{For every rational map $R$ such that $J(R)$ has positive Lebesgue measure, there exists a critical value $v\in J(R)$ with finite lower Lyapunov exponent.}

\smallskip
In  \cite{PrzyNonNegLya}, Przytycki proved that the Lyapunov exponents are 
non-negative for almost every point of $J(R)$ with respect to any 
finite invariant measure supported on the Julia set.
In the case of unimodal polynomials,   Levin,  Przytycki, and Shen 
(\cite{LevinPrzytyckiShen}) showed that the radius of convergence
 of the spectrum $\sigma(c)$ is always at least $1$ whenever $c$ 
belongs to the Julia set.  Moreover, they proved 
the stronger statement that the spectrum $\sigma(z)$ has radius of convergence at least $1$ not only for $z=c$ but  for Lebesgue almost every point $z$ of the Julia set $J(R)$.

 On the other hand, let us note that 
geometrically  divergent sequences are not N\"orlund regular. In particular, for a 
structurally stable $R$, if $\sigma(c)$ is N\"orlund regular then $c\in J(R)$. 

Given the data:
 
 \begin{enumerate}
 \item A sequence of points $\{z_n\}\subset \C$.
 \item A N\"orlund matrix associated to a sequence $\{q_n\}$.
\item A sequence $\{x_n\}$ such that $N(\{|x_n|\})=\{t_n\}$ has radius of 
convergence at least $1$.
\end{enumerate}
We associate the measure  $$\nu_\lambda=(1-\lambda)\sum_{n= 0}^\infty T_n 
\lambda^n,$$ where $T_n=q_nx_0\delta_{z_0}+
q_{n-1}x_1\delta_{z_1}+...+q_0x_n\delta_{z_n}.$ Then by Part (3) of Lemma 
\ref{lemma.Norlund}, the measures $\nu_\lambda$ form an analytic family of 
finite measures over the open unit disk. 

\begin{definition}
 Given a regular sequence $\{x_n\}$, we call a finite complex valued measure $\nu$ a \textit{Voronoi measure with respect
 to $N$ and the sequence}  $\{ z_n\}$ if there exists a complex sequence $\{\lambda_n\}$ with  $\lambda_n\rightarrow 1$ (and $|\lambda_n|< 1$) such that the sequence $\frac{\nu_{\lambda_n}}{\|\nu_{\lambda_n}\|}$ converges to $\nu$ in the $*$-weak topology, here $\|\nu_{\lambda_n}\|$ denotes the total variation of $\nu_{\lambda_n}.$  In order words, the  complex  projective 
 classes $[\nu_{\lambda_n}]$ converge to the complex projective class $[\nu]$, whenever $\nu\neq 0.$  If $\{x_n\}=\sigma(c)$ and $z_n=R^n(R(c))$ for a suitable critical point $c\in J(R)$, then we call  $\nu$ \textit{a Voronoi measure associated to $c.$}
\end{definition}

Let us mention that in \cite{LevinAnalytic},  complex projective classes of related measures were discussed.

Note that the support of a Voronoi measure associated to $c$ belongs to the individual postcritical set $P_c(R)$.

Voronoi measures always exist for every regular sequence $\{x_n\}$ and every sequence of points $\{z_n\}$. However, a Voronoi measure may not be uniquely determined by fixing either one of the sequences $\{x_n\}$ or $\{z_n\}$.  

Fix $\{x_n\}$ and $\{z_n\}$, according to the Consistence Theorem (see Theorem 17 in page 65 of the Hardy's book \cite{HardyDivergent}), if for two given N\"orlund matrices $N$ and $N'$  the N\"orlund averages  $\{t_n\}$ and $\{t'_n\}$ of $\{x_n\}$ converges then $\lim t_n =\lim t'_n.$ Assume that $N$ and $N'$ send the  sequence $\{x_n \phi(z_n)\}$ to convergent sequences, for every $\phi\in C(\overline{\bigcup \{z_n\}})$ the space of continuous functions on $\overline{\bigcup_{n= 0}^\infty \{z_n\}}.$  Then, by the Consistence Theorem, the set of Voronoi measures for $N$ coincides  with the set of Voronoi measures for $N'.$

Since there are Abel convergent  sequences that are not  Ces\`aro convergent, then it is possible that different N\"orlund matrices have different sets of Voronoi measures.

Now fix a N\"orlund matrix $N$, a regular $N$-sequence  $\{x_n\}$ and a sequence of points $\{z_n\}$, then the correspondence $\phi\mapsto \{\phi(z_n)\}$, for $\phi \in  C(\overline{\bigcup_{n= 0}^\infty \{z_n\}})$, is a continuous linear operator $T:C(\overline{\bigcup_{n= 0}^\infty \{z_n\}})\rightarrow \ell_{\infty}$. Hence, the family of measures $\omega_\lambda=\frac{\nu_\lambda}{\|\nu_\lambda\|}$ induces a family of functionals $L_\lambda(T(\phi))=\int \phi \, d\omega_\lambda$ which extends to an analytic uniformly bounded family of functionals on $\ell_\infty.$ Using arguments from functional analysis we have the following observation.

\begin{lemma}\label{l.uniqueness}
 Let $\{x_n\}$ be an $N$-regular sequence for the identity matrix $N$. Let $L_\lambda$ be the functionals on $\ell_\infty$ constructed above, assume that  the $*$-weak accumulation set of $L_{\lambda}$ is a single point set  then $\{x_n\}\in \ell_1$, for $\lambda \rightarrow 1$ and $|\lambda|<1$,
\end{lemma}

For convenience of the reader we prove this lemma in the following section.  

From Lemma \ref{l.uniqueness}, even in the simplest case, the accumulation set of $\{\omega_\lambda\}$, regarding each $\omega_\lambda$ as a  functional in $\ell_\infty$, is not unique whenever the sequence $\{x_n\}\notin \ell_1$.  

For instance, it is possible that there exists a bounded sequence $\{x_n\}$ such that the function given by $\phi(\lambda)=\frac{\sum_{n=0}^\infty x_n \lambda^n}{\sum_{n=0}^\infty | x_n\lambda^n|}$ is not continuous at $1$ from the interior of the unit disk. Then the set of Voronoi measures  for $\{x_n\}$ has more than one point regardless of the choice of the sequence $\{z_n\}$.  

Nevertheless,  it is possible to show that the space of all Voronoi measures, which are of Mergelyan type (see definition below), associated to a critical point  $c\in J(R)$ with bounded individual postcritical set, is finitely dimensional.

We call a finite complex valued measure $\mu$  a \textit{Mergelyan type measure}, 
or  an \textit{$M$-measure} for short, when its Cauchy transform $f_\mu(z)=\int_\C 
\frac{d\mu(t)}{t-z}$ is not identically $0$ outside the support of $\mu$.

For $K\subset \C$  compact, let $C(K)$ be the space of continuous functions on 
$K$  and $Rat(K)$ be the space of rational functions restricted to $K$  with poles outside of 
$K$.  If $Rat(K)$ is dense in $C(K)$ with respect to uniform convergence,  then any 
complex valued finite measure supported on $K$ is an $M$-measure. When 
either $m(K)=0$ or $\inf \{diam(W): W \textnormal{ component of } \C\setminus 
K\}>0$, by classical results, we get $\overline{Rat(K)}=C(K)$. For a more deep 
treatment of the theory see for example the book by T. Gamelin 
\cite{GamelinUniform}.

As a consequence, if $R$ is a map with a completely invariant Fatou
component,  then every finite complex valued non-zero measure supported on $J(R)$ is an $M$-measure. 
In particular, this is the case when $R$ is a polynomial.

Now we are ready to formulate our first main theorem.

\begin{theorem}\label{th.Norlundregular}
 Suppose that for a critical point $c\in  J(R)$, with bounded $P_c(R)$, the sequence $\{|
 \sigma_n(c)|\}$ is $N$-regular for a N\"orlund matrix $N$.
 If a  Voronoi measure $\nu$ associated to $c$ is an $M$-measure, then 
 $R$ is unstable. \end{theorem}

 The following theorem   deals with the situation when the sequence of measures $\{\nu_{\lambda_i}\}$  converges to zero in the $*$-weak topology for a suitable sequence $\lambda_i \rightarrow 1$. However, as in the case discussed by Avila, Levin, and Makienko, we need additional restrictions on the individual postcritical set. We say that a critical point $c$ and  the individual postcritical set $P_c(R)$
are \textit{separated by the Fatou set} if there exists a Jordan curve in $ F(R)$ separating $c$ and $P_c(R)$.
\begin{theorem}\label{th.NorlundTop} Under the conditions of Theorem \ref{th.Norlundregular}.
Assume that $\lambda_i$ is a sequence of real numbers with $\lambda_i<1$ and
 $\lambda_i\rightarrow 1$ for which $\nu_{\lambda_i}$ is $*$-weakly convergent to $0$. 
 Then $R$ is unstable whenever $c$ and $P_c(R)$ are separated by the Fatou 
 set. 
\end{theorem}

The separation condition is stronger than non-recurrence of critical points and is void for connected Julia sets. However, for disconnected Julia sets the two conditions are  closely related. For instance, if $J(R)$ is a Cantor set then 
every non-recurrent critical point in $J(R)$ satisfies the 
separation condition.  

Also, if the Fatou set $F(R)$ contains a completely invariant component and $J(R)$  is disconnected, then generically $J(R)$ contains uncountably many wandering single-point components. If a non-recurrent critical point $c$ is one of those wandering single-point components, then $c$ also satisfies the separation condition.

On the other hand, if a critical point $c$ belongs to a preperiodic (non-periodic) 
non single-point component, then $c$ satisfies the separation condition. 
This happens, for example, when $c$ belongs to  the boundary of 
a preperiodic but not periodic Fatou component.

Finally, let us note that, in general, if there is no completely invariant component 
but there is an infinitly connected Fatou component, then the set of buried 
points contains uncountably many single-point components. Then again, if a 
non-recurrent critical point $c$ is a buried single-point component of $J(R)$, then  also 
the separation condition holds. 

Nevertheless, the separation condition applies to a single critical point regardless of the behavior of other 
critical points.  

Let us note that Theorem \ref{th.NorlundTop} is complementary to Theorem \ref{th.Norlundregular} in the sense that if $\{x_n\}\in \ell_\infty$ then $\|\nu_\lambda\|=O(\|x_n\|_\infty)$ for real $\lambda$ and the assumption that all Voronoi measures are null implies that $\nu_\lambda\rightarrow 0$ for $\lambda\rightarrow 1$ in the $*$-weak topology.  However, it is possible that $\nu_\lambda$ converges to $0$ even when there are non-zero Voronoi measures, for instance consider a sequence $\{x_n\}\in \ell_1$.

Among other facts, in \cite{MakRuelle} it was established that if $c\in J(R)$ is non-recurrent and the series $\sum_{n=0}^\infty \sigma_n(c)$ is 
absolutely convergent, then $R$ is unstable. However, it is yet unclear whether 
a structurally stable map may  have a critical point $c\in J(R)$ with  $\sigma(c)\in 
\ell_\infty$. 

As a consequence of the previous theorem we answer this question positively by replacing the non-recurrent condition by a separation condition on the individual postcritical set $P_{c}$. Recall that a sequence $\{x_n\}\in \ell_\infty$ is \textit{Abel summable} to $L$ (maybe $\infty)$ whenever the limit 
$$\lim_{\lambda \rightarrow 1}\sum_{n= 0}^\infty \lambda^n x_n$$  for $ \lambda <1$ exists and is equal to $L.$ 
\begin{theorem}\label{bounded}
Let $c$ be a critical point in $J(R)$ with bounded $P_c(R)$  and spectrum $\sigma(c)\in \ell_\infty.$ Then:

\begin{enumerate}
\item The map $R$ is unstable whenever $P_c(R)$ has zero Lebesgue measure and $\sigma(c)$ is not Abel summable to $0.$
 \end{enumerate}
Even more, if $c$ and $P_{c}$ are separated by the Fatou set,  then each of the following conditions imply that $R$ is unstable.
\begin{enumerate}
 \item[2.]  The individual postcritical set $P_c(R)$ has measure $0$.
 \item[3.] The spectrum $\sigma(c)$ is convergent.
\end{enumerate}

\end{theorem}

Moreover, the condition that the Lebesgue measure of $P_c(R)$ is zero can be dropped in the case of  maps with a completely invariant Fatou domain.

\textbf{Acknowledgement.} The authors would like to thank M. Lyubich for useful discussions and to the referee for useful comments to a previous version of this paper. 

\section{Preliminary results.}

In this section we prove Proposition \ref{prop.motivation} and Lemma \ref{l.uniqueness} together with its corollaries. 

\begin{proof}[Proof of Proposition \ref{prop.motivation}.]

Given a stable map $R$ in the family $R_\lambda=R+\lambda$ for $|\lambda|<1$.  According to R. Ma\~n\'e, P. Sad, and D. Sullivan 
 (\cite{MSS}) there exists a holomorphic family $f_\lambda$ of quasiconformal automorphisms of $\C$ such that for  $\epsilon$ small enough and $\lambda$ with $|\lambda|<\epsilon$ we have 
 $$R_\lambda=f_\lambda\circ R\circ f_\lambda^{-1}.$$
 
 Let $F(z)=\displaystyle{\frac{\partial f_\lambda}{\partial \lambda}|_{\lambda=0}(z)}$ be the variation of $f_\lambda$.  Then $F$ is a 
 continuous function on $\C$. By a straightforward computation of the variation of $R_\lambda$
 we obtain $$F(R(z))-R'(z)F(z)=1.$$ Hence $F(R(c))=1$, and by induction $$F(R^n(R(c)))=(R^n)'(R(c))S_n.$$ We end up with $\left|S_n \right|\leq \sup_{z\in P_c(R)}|F(z)\sigma_n(c)|$. Since $F$ is bounded on  compact subsets of the plane, we are done.
 \end{proof}
 
 Corollary \ref{cor.motivation} is an immediate consequence of Proposition \ref{prop.motivation}. Now we prove Corollary \ref{cor.positiverat}.
 
\begin{proof}[Proof of Corollary \ref{cor.positiverat}]
  By  contradiction, assume that $R$ is structurally stable in $\{R_\lambda\}$. By assumption, the partial sums $S_n=\sum_{i=0}^n \sigma_i(c)$ satisfy 
  $S_n<S_{n+1}$ with $\limsup_n S_n>0$. Suppose that $\sigma(c)$ satisfies the condition (2) of the thrichotomy. Immediately we have $\sigma(c)\in \ell_1$ which contradicts Corollary \ref{cor.motivation}.
  Now suppose that $\sigma(c)$ satisfies the condition (3) of the thrichotomy.  Since $\liminf_n \sigma_n(c)\neq 0, \infty$, then $\displaystyle{\lim_{n\rightarrow \infty}} S_n=\infty$ which again contradicts Corollary \ref{cor.motivation}. \end{proof}

Now we briefly describe a large set of rational maps with real coefficients and non-negative derivative on the individual postcritical set $P_c(R)$ for a suitable critical point $c$.

Let $f$ be a real rational function which attains a finite minimum $m_0$ in the extended real line, then $f(z)-m_0$ is non negative with a zero in the extended real line. If $R(z)$ is a real primitive of $f(z)-m_0$  then $R(z)$ has a real critical point $c$  and $R'(z)=f(z)-m_0$ is non-negative on $P_c(R)$.
Note that  for $\gamma,\tau \in PSL(2,\mathbb{R})$, the map $\gamma \circ R \circ \tau$ also has the desired property.
To construct a real rational function $f$ as above fix the following data:

\begin{enumerate}
 \item A real polynomial of even degree $P(z)$, with positive leading coefficient. Here $P(z)$ may be a positive constant.
 
 \item A finite set $\{a_i\}$ of real points.
 
 \item A finite set  $\{b_i\}$ of positive real numbers.
 \item A finite set  $\{n_i\}$ of even integers numbers.
\item A finite set $\{k_i\}$ of integers greater than $1$.
\item A finite set $\{d_i\}$ of real numbers.
\item A finite set $\{w_i\}$ of non-real complex numbers. 
\end{enumerate}

Then the function $$f(z)=P(z)+\sum  \frac{b_i}{(z-a_i)^{n_i}} +\sum d_i \left(\frac{1}{(z-w_i)^{k_i}}+\frac{1}{(z-\overline{w}_i)^{k_i}}\right)$$ is the desired real rational function. The set of rational maps for which $R'\geq 0$ on $P_c(R)$ for a suitable critical point $c$ includes real maps with non-real critical points.
 
 \begin{proof}[Proof of Corollary \ref{Cor.AbelAver}]
By contradiction, if $R$ is a stable rational map in the family $\{R_\lambda\}$, then by assumption and Proposition 
\ref{prop.motivation} there exists a constant $C$ such that   $$\frac{|S_n|}{n}\leq \frac{C}{n}|\sigma_{n}(c)|.$$  
As the right hand side of the latter inequality converges to $0$,  the Ces\`aro averages of the spectrum converges  to $0$. This implies that the Abel averages of the spectrum converges to $0$, which is a contradiction. 
\end{proof}

 We finish the section with the proof of Lemma \ref{l.uniqueness}.
 
 \begin{proof}[Proof of Lemma  \ref{l.uniqueness}.]  Fix an arbitrary sequence of real numbers  $\lambda_i\rightarrow 1$ and $\lambda_i<1$. For each $i$ define the sequence $\alpha^i=\{\alpha_n^i\}$ given by $\alpha^i_n=\displaystyle\frac{\lambda_i^n x_n}{\sum_{n=0}^\infty |\lambda_i^n x_n|}$ so  $\alpha^i \in \ell_1.$ 
 If $b\in \ell_\infty$ is the sequence $\{b_n\}$ where $$b_n=
\begin{cases}
\frac{|x_n|}{x_n}, \, x_n\neq 0\\
0, \, x_n=0.
\end{cases},$$
 then $L_{\lambda_i}(b)=\sum_{n=0}^\infty \alpha^i_n b_n=1$.  So we identify each functional $L_{\lambda_i}$ with the sequence $\alpha^i.$ Thus every accumulation point of $L_{\lambda_i}$ in the $*$-weak topology as functionals on $\ell_\infty$ is a continuous non-zero functional.

 By assumption the sequence $L_{\lambda_i}$ converges to a non-zero functional in the $*$-weak topology on $\ell_\infty$. That is the same that the sequence $\{\alpha^i\}$ converges in the weak topology on $\ell_1$. Since $\ell_1$ is weakly complete and 
  the weak topology coincides with the strong topology, the sequence $\{\alpha^i\}$ converges 
 in norm to a non-zero element $\beta=\{\beta_n\}\in \ell_1$. In particular, we have 
  $$\lim_i\alpha_n^i=\beta_n.$$ Let  $k$ be such that $\beta_k\neq 0$ then
  $$\lim_{i\rightarrow \infty} \sum_{n=0}^\infty \lambda_i^n|x_n|=\frac{x_k}{\beta_k}.$$ Since  $\{\lambda_i\}$ is arbitrary, we conclude that  the limit $\displaystyle{\lim_{\lambda\rightarrow 1} \sum_{n=0}^\infty\lambda^n|x_n|}$ exists and is finite, therefore $\displaystyle{\sum_{n=0}^\infty |x_n|<\infty.}$
  
 \end{proof}

\section{Some background in dynamics and Poincar\'e series}

Most of the material in this section can be found in \cite{MakRuelle} (see also 
\cite{Makarxiv}). 

A rational map $R$ defines a complex push-forward map on $L_1(\C)$ 
with respect to the Lebesgue measure.  
This contracting endomorphism is called the \textit{complex Ruelle-Perron-Frobenius}, 
or the \textit{Ruelle operator} for short. The Ruelle operator is explicitly 
given by the formula
\[
R_*(\phi)(z)=\sum_{y\in R^{-1}(z)}\frac{\phi(y)}{R'(y)^2}=\sum_i 
\phi(\zeta_i(z))(\zeta'_i(z))^2, 
\] 
where $\zeta_i$ is any local complete  system of branches of the inverse of $R$. 
The \textit{Beltrami operator} 
$Bel:L_\infty(\C) \to L_\infty(\C)$ given by  
$$
Bel(\mu)=\mu(R)\frac{\overline{R'}}{R'}
$$ 
is dual to the Ruelle operator  acting on $L_1(\C)$. 
The fixed point space $Fix(B)$ of the Beltrami operator is called the 
\textit{space of invariant Beltrami differentials}.

Every element $\mu \in L_\infty(\C)$ defines a continuous function on $\C$ via 
$$
F_\mu(a)=a(a-1)\int_\C  \frac{\mu(z)}{z(z-1)(z-a)}|dz|^2, 
$$ 
which is called the \fd{normalized potential for $\mu$}. 
By convenience, we write 
$\gamma_a(z)=\frac{a(a-1)}{z(z-1)(z-a)}$ so  we get 
$F_\mu(a)=\int \gamma_a(z) \mu(z) |dz|^2$.

The following statement appears as Lemma 5 and Remark 6 in \cite{MakRuelle}. 

\begin{lemma}\label{lm.struc.invariant}
 Let $R$ be a structurally stable rational map. 
Then for every critical value $v_i$ there exists an invariant 
Beltrami differential $\mu_i$ such that
 $F_{\mu_i}(v_j)=\delta_{ij}$, the Kroenecker delta  function.
 \end{lemma}

Now we present the formal relation of the Poincar\'e-Ruelle series. 
 
\begin{definition} The \textit{Poincar\'e-Ruelle} series are
\begin{itemize}
 \item 
$$
B_a(z)=\sum_{n=0}^\infty (R_*)^n(\gamma_a(z)),
$$ 
\item  
$$
A_a(z)=\sum_{n=0}^\infty \frac{1}{(R^n)'(a)}\gamma_{R^n(a)}(z).
$$
\end{itemize}
\end{definition}

The lemma below gives a formal relation between both Poincar\'e-Ruelle series. 
The proof can be found in Proposition 7 of  \cite{MakRuelle} and further details are contained in 
\cite{Makarxiv}.
 
\begin{lemma}\label{lem.ratcrit} Let $R$ be a rational map with simple critical points 
$c_i$  fixing $0,1$, and $\infty$.  Set $v_i=R(c_i)$ and let $a$ be a value such that 
$\bigcup_n\{R^n(a)\}$ does not contains critical points. Then we have the following 
formal relation between the above series 
 \begin{equation}\label{eq.Poincare}
 B_a(z)=A_a(z)+\sum_{i} \frac{1} {R''(c_i)}A_a(c_i)\otimes B_{v_i}(z),\tag{*}
 \end{equation}
 here $\otimes$ is the formal Cauchy product. 

\end{lemma}

 Hence, we have (see again Proposition 7 in \cite{MakRuelle}) \begin{align}(R_*)^n(\gamma_{a}(z))&=\frac{1}{(R^n)'(a)}\gamma_{R^n(a)}(z)+ \sum_i \frac{1} {R''(c_i)} \left[ \frac{\gamma_{R^{n-1}(a)}(c_i)}{(R^{n-1})'(a)}
                                \gamma_{v_i}(z)+ \right.\\
                                &\left. + \frac{\gamma_{R^{n-2}(a)}(c_i)}{(R^{n-2})'(a)} R_*(\gamma_{v_i}(z))+\cdots
+\gamma_a(c_i)(R_*)^{n-1}(\gamma_{v_i}(z)) \right].\end{align}

To the formal series $A_a(z)$ and $B_a(z)$ involved in Equation 
\eqref{eq.Poincare} we associate a formal Abel series parameterized by the unit disk as follows. For $|\lambda|<1$, write
$$A_a(z,\lambda)=\sum_{n=0}^\infty \frac{\lambda^n}{(R^n)'(a)}\gamma_{R^n(a)}(z)$$ and
$$B_a(z,\lambda)=\sum_{n=0}^\infty \lambda^n(R_*)^n(\gamma_{a}(z)).$$

We have the following lemma.

\begin{lemma}\label{lem.convergence}
Let $R$ be a structurally stable rational map, $c\in J(R)$ be a critical point with bounded $P_c(R)$ and $v=R(c)$. Assume  the
spectrum $\sigma(c)$ has  radius of convergence $r>0$,  then  for any complex number $\lambda$ with $|\lambda|<r$ we have the following. 
\begin{enumerate}
 \item The series $A_v(z,\lambda)$ is absolutely 
 convergent almost everywhere with respect to $z$ and is an integrable function holomorphic off $P_c(R)\cup\{0,1,\infty\}.$
 \item Let $\tilde{c}$ be a critical point. The numerical series $A_v(\tilde{c},\lambda)$ 
 is absolutely convergent.
 \item  The series $B_a(z,\lambda)$ is absolutely 
 convergent almost everywhere and is an integrable function for every $|\lambda|<1$ and every $a\in \C$.
\end{enumerate}
 Furthermore, each of the series above defines a holomorphic function with respect to 
 $\lambda$ for $|\lambda|<r.$
 \end{lemma}

 \begin{proof}
 
 By assumption the series $\sum \lambda^n \sigma_n(c)$ is absolutely convergent 
 for $|\lambda|<r$ and defines a holomorphic function with respect to $\lambda$ in the disk $|\lambda|<r.$
 
 Part (1). As there exists a constant $C$ such that $\int |\gamma_a(z)||dz|^2\leq C |a\ln|a||$ (see for example the books by 
 Gardiner-Laki\v{c} \cite{GardLakic} and Krushkal  \cite{KrushkalQCRiem}) we get
 \begin{align}\int_\C |A_v(z,\lambda)||dz|^2\leq \sum 
 |\lambda^n \sigma_n(c)|\int_\C\left|\gamma_{R^n(v)}(z)\right| |dz|^2 \\
 \leq C\sum |\lambda^n\sigma_n(c)R^n(v)\ln|R^n(v)||.
\end{align}
Since $P_c(R)$ is bounded for every $|\lambda|<r$, the last expression is absolutely convergent. Hence $A_\nu(z,\lambda)$ converges in the $L_1$ norm and is a holomorphic integrable function outside   $P_c(R)\cup 
\{0,1,\infty\}.$  By the mean value theorem 
the partial sums of the series $A_v(z,\lambda)$ are uniformly bounded on compact sets outside  $P_c(R)\cup 
\{0,1,\infty\}.$ Hence the series $A_\nu(z,\lambda)$ converges uniformly on compact sets off  $P_c(R)\cup 
\{0,1,\infty\}.$ 

Part (2). If $\tilde{c}\notin P_c(R)$ then by Part (1) we are done. Otherwise, we have 
$\tilde{c}\in P_c(R)$. Given $\epsilon>0$, let $U_\epsilon$ be the $\epsilon$ 
neighborhood of $\tilde{c}$. Let $n_i$ be such that $R^{n_i}(v)\in U_\epsilon.$ 
As in the arguments in Part(1), it is enough to estimate the expression $$\left|
\sum_i  \lambda^{n_i}\sigma_{n_i}(c)\gamma_{R^{n_i}(v)}(\tilde{c})\right|.$$ 
Note that for $z\in 
U_\epsilon$, we have
$R'(z)=(z-\tilde{c})R''(\tilde{c})+O(|z-\tilde{c}|^2).$ Thus, we get
$$\left|\frac{1}{R^{n_i}(v)-\tilde{c}}\right|\leq \left|\frac{R''(\tilde{c})+O(|R^{n_i}(v)-\tilde{c}|)}{R'(R^{n_i}(v))}\right|\leq M  \left|\frac{1}{R'(R^{n_i}(v))}\right|$$  and $$\left|\gamma_{R^{n_i}(v)}(\tilde{c})\right|\leq 
M\left|\frac{1}{R'(R^{n_i}(v))}\cdot \frac{R^{n_i}(v)(R^{n_i}(v)-1)}{\tilde{c}
(\tilde{c}-1)}\right|\leq M_1\left|\frac{1}{R'(R^{n_i}(v))}\right|,$$ where $M$ and $M_1$ are suitable constants  depending on $\epsilon$ and $\tilde{c}$. As a result of the previous computation we obtain

\begin{align}\left|\sum_i  \lambda^{n_i}{\sigma_{n_i}(c)}\gamma_{R^{n_i}(v)}(\tilde{c})\right|&\leq M_1\sum_i  \left| \frac{\lambda^{n_i}{\sigma_{n_i}(c)}}{R'(R^{n_i}(v))}\right| \\
&\leq \frac{1}{|\lambda|}M_1\sum_i |\lambda^{n_i+1}
\sigma_{n_i+1}(c)|<\infty,
\end{align} for $0<|\lambda|<r.$

Part (3). Since $\|R_*(f)\|_{L_1}\leq \|f\|_{L_1}$ holds for any $f\in L_1(\C)$, then for
every $|\lambda|<1$, the series $B_a(z,\lambda)$ is an integrable function 
and so converges absolutely almost everywhere for $a\in \C$.\end{proof}

 We have the following immediate consequence. 
 
 \begin{corollary}\label{cor.posrad} Let $c$ be a critical point of a rational map $R$ with bounded $P_c(R).$
  If the spectrum $\sigma(c)$ of a critical point $c$ has radius of convergence $r>0$ 
  then we can rewrite Equation \eqref{eq.Poincare}  as 
  \begin{equation}\label{eq.posrad} 
  B_{R(c)}(z,\lambda)=A_{R(c)}(z,\lambda)+\lambda \sum_i \frac{1}{R''(c_i)}  
  A_{R(c)}(c_i,\lambda)\cdot B_{v_i}
 (z,\lambda),\tag{**}
  \end{equation}
  for every $|\lambda|<r$ and almost every $z\in \C.$ 
 \end{corollary}

 \begin{proof}
This comes from the previous lemma  and the Cauchy product theorem.\end{proof}

 \section{Proofs of the main theorems}
 
 Now we are ready to proof Theorem \ref{th.Lyapunovnon}, Theorem \ref{th.Norlundregular},  Theorem \ref{th.NorlundTop} and Theorem \ref{bounded}. 
 
 \subsection{The Abel case}
 
 Our plan is first to prove these theorems in the case where 
a spectrum has radius of convergence at least $1$ and taking the identity matrix as N\"orlund matrix, this is what we call the
 \textit{Abel case}. We will consider general matrices in the next subsection. Accordingly, a Voronoi measure associated  to the identity matrix, a regular sequence $\{x_n\}$ and a sequence of points $\{z_n\}$ will be called an \textit{Abel measure}. Again, as in the Voronoi case, we say that an Abel measure is associated to a point $a$ whenever $\sigma(a)=\{x_n\}$ and $z_n=R^n(a).$ In this situation, the measure $\nu_\lambda$ has the following simple form $$\nu_\lambda=(1-\lambda)\sum_{n=0}^\infty \sigma_n(a)\lambda^n\delta_{R^n(a)}.$$
 
 \begin{lemma}\label{lemma.31} Let $c$ be a critical point with bounded $P_c(R)$ and such that the 
 $\sigma(c)$  has radius of convergence at least  $1$.
 Let $\nu_0$ be a finite complex valued measure. Let  $\{\lambda_i\}$ and $\{r_i\}$ be sequences of complex numbers, with $|\lambda_i|<1$, so that $r_i \nu_{\lambda_i}$ converges $*$-weakly to
 $\nu_0$.  If $R$ is structurally stable then for any critical point $\tilde{c}\neq c$ the sequence $\{r_iA_{R(c)}(\tilde{c},
 \lambda_i)\}$ is bounded and convergent. Finally, for the critical point $c$ the sequence $\{r_i(1-\frac{1}{R''(c)} A_{R(c)}(c,\lambda_i))\}$ is bounded and convergent.
 \end{lemma}

 \begin{proof}
Let $\{c_j\}_{j=1}^{2 deg(R)-2}$ be the set of critical points with $c_1=c$ and $v_j=R(c_j)$ their respective critical values. For a critical value $w$ take an invariant Beltrami differential $\mu_w$ as in Lemma \ref{lm.struc.invariant}.

As $\mu_w$ is invariant, we get
$$\int_\C \mu_w(z)B_{v_j}(z,\lambda)|dz|^2=\frac{1}{1-\lambda}F_{\mu_w}(v_j)=\begin{cases}
\frac{1}{ 1-\lambda},\, w=v_j \\
0  \, \textnormal{ otherwise.}\end{cases}$$

Integrating Equation \eqref{eq.posrad} in Corollary \ref{cor.posrad} with respect to
$\mu_w$ yields
\begin{align*}\int_\C\mu_w(z) B_{v_1}(z,\lambda)|dz|^2&=\int_\C\mu_w(z) A_{v_1}(z,\lambda)|dz|^2+\\ &\hskip.21cm+\lambda \sum_j \frac{1}{R''(c_j)}  
  A_{v_1}(c_j,\lambda)\cdot \int_\C \mu_w(z)B_{v_j} 
 (z,\lambda)|dz|^2.
 \end{align*}

After multiplying on both sides of the latter equation by $r_i(1-\lambda_i)$ and taking $\lambda=\lambda_i$, we obtain

\begin{equation}\label{eq.lema3.1} r_i(1-\lambda_i)\int \mu_w(z) A_{v_1}(z,\lambda_i)|dz|^2= r_i\left[F_{\mu_w}(v_1)-\frac{\lambda_i}{R''(c)}A_{v_1}
(c,\lambda_i)\right].\end{equation}

On the other hand as we have $(1-\lambda_i)A_{v_1}(z,\lambda_i)=\int \gamma_a(z) d\nu_{\lambda_i}(a),$ 
then applying Fubini's theorem we get 
$$r_i(1-\lambda_i)\int \mu_w(z) A_{v_1}(z,\lambda_i)|dz|^2=r_i \int d\nu_{\lambda_i}(a)\int 
\gamma_a(z)\mu_w(z)|dz|^2$$
$$=\int 
F_{\mu_w}(a)r_id\nu_{\lambda_i}(a).$$ Since $F_{\mu_w}$ is continuous on $\C$ and $P_c(R)$ is bounded, we have   $$\lim_{i\rightarrow \infty}\int 
F_{\mu_w}(a)r_id\nu_{\lambda_i}(a)=\int F_{\mu_w}
(a)d\nu_0(a),$$ and then returning to (\ref{eq.lema3.1}), by the choice of $\mu_w$ we conclude the proof. \end{proof}

\begin{proposition}\label{prop.strucintegrab}
Let $R$ be structurally stable and
$\nu_0$ be as in Lemma \ref{lemma.31}. Then $\phi(z)=\int \gamma_a(z)d\nu_0(a)$ is a well-defined integrable function that satisfies $R_*(\phi(z))=\phi(z).$
\end{proposition}
 
\begin{proof} First, again since $P_c(R)$ is bounded and using Fubini's theorem we get
\begin{align*}\int |\phi(z)||dz|^2&\leq  \int |d\nu_0(a)|\int |\gamma_a(z)||dz|^2 \\
&\leq M\int |a|	|\ln |a|||d\nu_0(a)|<\infty.\end{align*} 
Therefore $\phi$ is integrable.
Applying $r_i(1-\lambda_i)[I-\lambda_iR_*]$ to Equation 
\eqref{eq.posrad} in Corollary \ref{cor.posrad}  with $\lambda=\lambda_i$, 
and then using the resolvent equation $$(Id-\lambda R_*)\circ(\sum_{n=0}^\infty\lambda^n(R_*)^ n)=Id,$$  
we obtain

\begin{align*}r_i(1-\lambda_i)\gamma_{v_1}(z)&=r_i(1-\lambda_i)\left[Id-\lambda_i\cdot 
R_*\right](A_{v_1}(z,\lambda_i)\\
&+r_i(1-\lambda_i)\lambda_i \sum_j \frac{1}{R''(c_j)} A_{v_1}(c_j,
\lambda_i)\cdot \gamma_{v_j}(z)).\end{align*}
By Lemma \ref{lemma.31}, we have that $\displaystyle{\lim_{i\rightarrow \infty}}\|r_i(1-\lambda_i)\left[Id-\lambda_i\cdot 
R_*\right]A_{v_1}(z,\lambda_i)\|_{L_1}=0$. 

Rearranging, we get  \begin{align*}r_i(1-\lambda_i)[Id-\lambda_i R_*] A_{v_1}(z,\lambda_i)&=[Id-R_*]
[r_i(1-\lambda_i)A_{v_1}(z,\lambda_i)]\\ & +R_*[r_i(1-\lambda_i)^2A_{v_1}
(z,\lambda_i)].\end{align*}

Finally $$\|R_*(r_i(1-\lambda_i)^2 A_{v_1}(z_1,\lambda_i))\|_{L_1}\leq |1-\lambda_i|\|r_i(1-\lambda_i)A_{v_1}(z,\lambda_i)\|_{L_1}$$
But the latter converges to $0$ as $i$ tends to $\infty$ since $\|r_i(1-\lambda_i)A_{v_1}(z,\lambda_i)\|_{L_1}$ is bounded. 
Since $\phi(z)=\displaystyle{\lim_{i\rightarrow \infty}} r_i(1-\lambda_i) A_{v_1}(z,\lambda_i)$
holds almost everywhere, we are done. \end{proof}

\begin{lemma}\label{lm.nonzero}
 Under conditions of  Proposition \ref{prop.strucintegrab}, assume that  measure  $\nu_0$ is an $M$-measure, then $\phi(z)=\displaystyle{\int \gamma_z(z)d\nu_0(a)}$ is non-zero identically on $\C\setminus P_c(R).$ 
\end{lemma}

\begin{proof} Since  $\nu_0$ is an $M$-measure, then we can think that  $\nu_0$ is a non-zero Abel measure. By assumption $R$ is structurally stable, then we can assume that $\sigma(c)\notin \ell_1$ by Proposition 9 and Proposition 10 (2) in \cite{MakRuelle}. A direct computation gives $\int \frac{f(R)}{R'}d\nu_0=\int f d\nu_0$ for every continuous function $f$. 
 
 Now we proceed by contradiction, if $\phi(z)\equiv 0$ on $\C\setminus P_c(R)$, then by assumption $$\int  \frac{1}{z-a}d\nu_0(a)=\frac{A}{z}+\frac{B}{z-1},$$ for suitable $A$ and $B$ complex numbers not both equal to $0$. In other words, if $\nu_1=\nu_0-(A\delta_0+B\delta_1) $, where $\delta_a$ denotes the delta measure at the point $a,$ the function $\Phi(z)=\int \frac{1}{z-a}d\nu_1=0$ for $z\in \C\setminus P_c(R)$.
 
 For $z\in \C \setminus P_c(R)$ we have  $$0=R_*(\Phi(z))=\int R_*\left(\frac{1}{z-a}\right)d\nu_1(a)$$ by part 1 of Lemma 5 in \cite{MakRuelle} we have 
 $$R_*(\Phi(z))=\int \left(\frac{1}{R'(a)(z-R'(a))}+\sum_i \frac{1}{R''(c_i)(c_i-a)(z-R(c_i))}\right)d\nu_1(a)$$ $$=\int \frac{1}{R'(a)(z-R(a))} d\nu_1(a)+ R_1(z)$$
where $R_1$ is a rational function with simple poles only in critical values of $R$.
But, since $\int \frac{1}{R'(a)(z-R(a))}d\nu_0(a)=\int \frac{1}{z-a}d\nu_0(a)$ we conclude $$\int \frac{1}{R'(a)(z-R(a)} d\nu_1(a)=\int \frac{1}{R'(a)(z-R(a))}d\nu_0(a)-\frac{A}{R'(0)z}-\frac{B}{R'(1)(z-1)}$$$$=\frac{A}{z}\left(1-\frac{1}{R'(0)}\right)+\frac{B}{z-1}\left(1-\frac{1}{R'(1)}\right):=R_0(z)$$ is a rational function with $0=R_*(\Phi(z))=R_0(z)+R_1(z)$. Then by the Residue theorem either $R'(0)=1$ or $R'(1)=1$ which contradicts the structural stability of $R.$
\end{proof}
 The following theorem is the Abel version of Theorem \ref{th.Norlundregular}. 
 
 \begin{theorem}\label{TheoremA}
 Let $c\in J(R)$ be a critical point so that $P_c(R)$ is bounded and $\sigma(c)$ has radius of convergence at least $1$. Then $R$ is unstable whenever the Abel measure associated to $c$ is an $M$-measure. 
 \end{theorem}

 \begin{proof} Let $\nu_0$ be an Abel measure associated to $c$ then there is a sequence of $\lambda_i\rightarrow 1$ such that $\frac{\nu_{\lambda_i}}{\|\nu_{\lambda_i}\|}$ converges $*$-weakly to $\nu_0$. Since $\nu_0$ is an M-measure then $\nu_0\neq 0$.
 
Assume that $R$ is structurally stable then, by Lemma \ref{lm.nonzero}, we have that   $\phi(z)=\int \gamma_a(z)d\nu_0(a)\not\equiv 0$ on 
$\C\setminus P_{c}(R)$, and $R_*(\phi)=\phi$ by Proposition \ref{prop.strucintegrab}. Hence by  Corollary 12 in \cite{MakRuelle} and Lemma 3.16 of   
\cite{Mc1} (see also \cite{Adamrigidity}),  $R$ is a flexible Latt\`es map. This is a 
contradiction with the structural stability of $R$. \end{proof}
 
 Now we are ready to  prove  Theorem \ref{th.Lyapunovnon}.
 
 \begin{proof}[Proof of Theorem \ref{th.Lyapunovnon}]
 Since $R(x)$ has non-negative lower Lyapunov exponent then $\sigma(x)$ has radius of convergence at least 1. Because 
 $R'(z)\geq 0$ on $\bigcup_{n=1}^\infty R^n(x)$ then, for each  $0\leq \lambda <1$, the measure $$\omega_\lambda=\frac{\displaystyle{\sum_{n=0}^\infty} \lambda^n \sigma_n(x) \delta_{R^n(x)}}{\displaystyle{\sum_{n=0}^\infty} \lambda^n \sigma_n(x)}$$
  is a
 probability measure. Let $\omega$ be an accumulation point of $\{\omega_\lambda\}$, then $\omega$ is a probability measure. Since $\sigma(x)\notin \ell_1$, then a straightforward calculation gives 
 \begin{displaymath}\label{star}\tag{2} \int \frac{\phi(R)}{R'}d\omega=\int \phi d\omega 
 \end{displaymath}  for every continuous function $\phi$ on the support of $\omega.$  
 
 Assume that $R$ is stable in the family $R_\lambda$, then by Proposition \ref{prop.motivation} there exist a function $F$, continuous on the plane, with 
 $$\frac{F(R)(a)}{R'(a)}-F(a)=\frac{1}{R'(a)}$$
 integrating the latter equation with respect to $\omega$ leads to a contradiction with equation (\ref{star}) above. Since the right side becomes $0$ whereas the left side becomes $1.$
 
 If $x$ is a critical point, then $\sigma(x)$ does not belong to $\ell_1$ by Corollary \ref{cor.motivation}.
  
 \end{proof}

\textbf{Remark}. Indeed  the non-negative condition on the orbit of $x$ is not necessary whenever $\omega$ is non-zero, $\sigma(x)\notin \ell_1$ and $R(x)$ has non-negative lower Lyapunov exponent. 
 
Theorem \ref{th.Lyapunovnon} gives examples of rational maps having non-zero Abel measures.  Now we give three criteria for a complex sequence to allow a non-zero Abel measure.  The first
 is elementary and states that if the arguments of the sequences 
 $\lambda^n_i a_n$ are close enough to $0$ then this sequence has a non-zero Abel measure with respect to every  rational map $R$ and  $z\in \bar{\C}$. More precisely, we have the following lemma.
 
 \begin{lemma}
 Let $\{a_n\}$ be a complex sequence with radius of convergence at least $1$.
 Assume  there exist $\alpha<1$ and a complex sequence $\{\lambda_i\}$ converging 
 to $1$ with $|\lambda_i|<1$ such that  
 $$\left| \lambda_i^n a_n-|\lambda_i^n a_n|\right|\leq \alpha |\lambda_i^n a_n|.$$
 Then for every rational map $R$ and every point $z\in \C$, there exists a non-zero Abel  measure with respect to $\{a_n\}$ and $z$.
 \end{lemma}
 \begin{proof}
 Fix $R$ and $z\in \C$. Then every $*$-weak limit of the family of the probability 
 measures 
 $$w_\lambda=\frac{\sum |\lambda|^n |a_n|\delta_{R^n(R(z))}}{\sum |\lambda^n||
 a_n|}$$ is a probability measure. Write $$u_\lambda=\frac{\sum \lambda^n a_n
 \delta_{R^n(R(z))}}{\sum |\lambda^n||a_n|},$$ so that $u_\lambda$ is a family of 
 complex valued measures absolutely continuous with respect to $w_\lambda.$ For the 
 sequence $\{\lambda_i\}$ assume by contradiction that $u_{\lambda_i}$ converges $*$-weakly to $0$.
Define $X=\overline{\bigcup R^n(z)}$ and let $1_X$ be the characteristic function on $X$. Notice that the supports of $w_\lambda$ and $u_\lambda$ belong to $X$. Now 
the inequality
$$ 1=\lim_{\lambda_i \rightarrow 1}
\left|\int 1_X du_{\lambda_i}-\int 1_X dw_{\lambda_i}\right|$$$$
 =\lim_{\lambda_i \rightarrow 1}
\frac{|\sum_n \lambda^n_i a_n-\sum_n |\lambda^n_i a_n||} {\sum_n |\lambda^n_i a_n|}$$$$
 \leq \lim_{\lambda_i \rightarrow 1}\frac{\sum_n| \lambda^n_i a_n- |\lambda^n_i a_n||}{\sum_n |\lambda^n_i a_n|}$$$$
 \leq \alpha <1 $$
establishes a contradiction. \end{proof}

 The second criterion is connected with the $L_1$ norm of the function $A_z(\lambda)$ 
 on  $\C\setminus P_c(R)$. Indeed, we prove a more general statement. 
 
 \begin{lemma}\label{lm.weakaccum} Let $K\subset \C$ be compact and let $\nu_i$ be a bounded sequence
 of complex valued measures on $K$. If we assume $$\limsup_i \int_{\C} \left|\int_\C 
 \gamma_a(z)d\nu_i(a)\right||dz|^2>0,$$ then there exists a $*$-weak accumulation 
 point $\nu_0$ which is not null. The reciprocal is also true.
 \end{lemma}

 \begin{proof}
 According to a well known result on quasiconformal theory (see for example  
F. Gardiner and N. Laki\v{c} \cite{GardLakic} or S. L. Krushkal \cite{KrushkalQCRiem}) the operator 
 $T:L_\infty(\bar{\C})\rightarrow C(K)$ given by 
 $$T(\mu)(a)=\int_\C \gamma_a(z)\mu(z)|dz|^2,$$ which maps $\mu$ to $F_\mu|_K$, is 
 continuous and compact. The same is true for the dual operator $T^*: M(K)\rightarrow (L_\infty(\C))^*$ 
 given by $$T^*(m)(z)=\int_K \gamma_a(z)dm(a),$$ which is continuous and compact.
 
 By Fubini's theorem we have $rank(T^*)\subset L_1(\C)$ and each $T^*(m)(z)$  is  
 holomorphic off of $K\cup \{0,1,\infty\}.$ Hence $T^*:M(K)\rightarrow L_1(\C)$ is 
 a compact operator.  Now, by assumption, passing to a subsequence, we have that
 $$f_i(z)=\int_K \gamma_a(z)d\nu_i(z)=T^*( \nu_i)(z)$$ converges in norm to a
 non-zero $f$ in $L_1(\C).$
 
 Note that we have $\partial_{\bar{z}}f_i=\nu_i$ in the sense of distributions. If 
 $\nu_i$ converges $*$-weakly to $0$, then, by continuity, the integrability of $f$, and 
 an application of  Weyl's lemma we get $f=0$, which is a contradiction, so $\nu_i$ 
 cannot converge $*$-weakly to $0$. 

 Reciprocally, if a measure $\nu_0\neq 0$ is a $*$-weak limit of $\nu_i$ then 
 $T^*(\nu_0)\not\equiv 0$ almost everywhere on $\C$ and $T^*(\nu_i)$ converges to $T^*(\nu_0)$ 
 in norm.\end{proof}
 
 The last criterion is formulated for bounded sequences and follows from Proposition \ref{prop.motivation}. In this
 situation, we consider the 
 Abel sum $S(\lambda)=\sum a_n \lambda^n$. If the Abel averages $A(\lambda)=(1-\lambda)S(\lambda)$ is not continuous at $1$ from the left, then there exists a non-zero Abel measure. However, if  $\limsup_{\lambda < 1} |S(\lambda)|$ is bounded then $\lim_{\lambda\rightarrow 1^{-}}A(\lambda)=0$. In this situation we have the following existence lemma. 
 
 \begin{lemma}\label{lem.36}
  If $R$ is structurally stable and $z_0\in \C$ has infinite bounded forward orbit such that $\sigma(z_0)\in \ell_\infty$. Then there is a non-zero Abel measure
  whenever $\sigma(z_0)$ is not Abel summable to a finite limit. Moreover, if $z_0$ is a critical value, then the Abel measure is non-zero whenever $\sigma(z_0)$ is not Abel summable to zero.
 \end{lemma}

 \begin{proof}
Let $F$ be the continuous function on the plane constructed in the proof of Proposition \ref{prop.motivation}, then for every $n$ we have 
  $$\frac{F(R^n(z_0))}{(R^n)'(z_0)}=F(z_0)-1+(1+\sigma_1(z_0)+...+\sigma_n(z_0))=F(z_0)-1+S_n(z_0).$$ 
  Multiplying by $\lambda^n$ and adding with respect to $n$ we get
  $$\sum_{n=0}^\infty \lambda^n F(R^n(z_0))\sigma_n(z_0)=\frac{F(z)-1}{1-\lambda}+\sum_{n=0}^\infty \lambda^n S_n(z_0),$$ and so $$(1-\lambda)\sum_{n=0}^\infty \lambda^n F(R^n(z_0))\sigma_n(z_0)=F(z_0)-1+\sum_{n=0}^\infty\lambda^n \sigma_n(z_0).$$

  Since $\sigma(c)\in \ell_\infty$ and the only
  Abel measure is $0$, then the lefthand side of the previous equation
  converges to $0$ as $\lambda\rightarrow 1^-$, which contradicts the fact that the righthand side is not continuous at $1$ from the left. If $z_0$ is a critical value, then by Proposition \ref{prop.motivation} we have $F(z_0)=1$ and now we apply the formulae above to finish the proof.
 \end{proof}

We note that real infinitely renormalizable quadratic polynomials $f_v(z)=z^2+v$, with bounded combinatorics and definite moduli bounded below by  $m$, serve as examples of maps with $\sigma(c)\in \ell_\infty$ for the critical point $c$. Indeed, let $\mathcal{R}_{p_n}(f_v)(z)=\alpha_n f_v^{p_n}(z)\alpha_n^{-1}$  be the $n$-th level of renormalization for a suitable $\alpha_n>1$ (see \cite{Mc2}). By Theorem 5.8 in  \cite{Mc1}, the renormalizations  $\mathcal{R}_{p_n}(f_v)$ converge to a quadratic-like map $g_0$ when $n$ tends to infinity. Let $v_n=R_{p_n}(f_v)(0)$, then we have $R'_{p_n}(f_v)(v_n)=(f_v^{p_n})'(f_v^{p_n-1}(v))$.   By Lyubich's Theorem II in \cite{LyuI}, the map $f_v^{p_n-1}$ defines a diffeomorphism of uniformly bounded distortion with a constant depending only on $m$. Then $(f_v^{p_n})'(v)$ is comparable with $(f_v^{p_n})'(f_v^{p_n-1}(v))$ by a constant $K$ which depends only on $m$.  Now, for sufficiently big $p$,  by Theorem 8.1 in \cite{Mc2}, for a constant $C$ depending on $m$, we have that  $|(f_v^j)'(v)|>C$ for $j=1,...,p-1$. 

 To prove the Abel version of  Theorem \ref{th.NorlundTop} we need some additional preparation. Recall that a positive 
 Lebesgue measurable subset $W$ of $\bar{\C}$ is called \textit{wandering} if $R^{-n}
 (W)$  forms a family of pairwise almost disjoint sets with respect to the Lebesgue 
 measure. The union $D(R)$ of all wandering sets is called  the \textit{dissipative set}, 
 its complement $\bar{\C}\setminus D(R)$ is the \textit{conservative set}. A 
 Fatou component $U$ belongs to $D(R)$ precisely when $U$ is not an invariant rotational
 component.
 
 We start the proof of Theorem \ref{th.NorlundTop} with  the following lemma which is reminiscent of Lemma  \ref{lm.struc.invariant} and the arguments of the proof of Lemma \ref{lemma.31}.
 
 \begin{lemma}\label{lm.measures}
  Let $R$ be a structurally stable map. If the measures $\nu_{\lambda_i}$ converges $*$-weakly to $0$ for a suitable sequence of reals $\lambda_i \rightarrow 1$ and a critical point $c\in J(R)$, then 
$$
\lim_{\lambda_i\rightarrow 1}A_{R(c)}(\tilde{c},\lambda_i) = 
\begin{cases}
  \phantom{-}0  & \text{if}\ \tilde{c}\neq c \\
  R''(c)           & \text{if}\ \tilde{c}=c.
\end{cases}
$$
 \end{lemma}

 \begin{proof}
 By assumption,  $\nu_{\lambda_i}$ converges $*$-weakly to 0, then 
 $\sup\|\nu_{\lambda_i}\|<\infty$, where $\|\nu\|$ is the total variation of $\nu$.
 Since $P_c(R)$ is bounded, by Fubini's theorem we have 
 $$\|(1-\lambda)A_{R(c)}(z,\lambda_i)\|_{L_1}\leq M \sup_{a\in P_c(R)} |a||\ln|a|| \sup_i \|\nu_{\lambda_i}\|\leq \infty$$
for a suitable constant $M$ depending on $P_c(R).$
For every critical point $c_j$ we define $$D(c_j,\lambda_i)=\begin{cases}
\phantom{-}\left|\lambda_i \frac{A_{R(c)}(c_j,\lambda_i )}{R''(c_j)}\right |   & \text{if}\ c_j\neq c \\
\left|\lambda_i \frac{A_{R(c)}(c_j,\lambda_i )}{R''(c_j)}-1\right|          & \text{if}\ c_j=c.
\end{cases}$$ 
By structural stability,  Equation \eqref{eq.posrad} of Corollary  \ref{cor.posrad} and 
Lemma \ref{lm.struc.invariant} there are constants $M_1$ and $M_2$ such that for every $i$ 
$$M_1 \max_j\left(D(c_j,\lambda_i)\right )\leq \|(1-\lambda)A_{R(c)}(z,\lambda_i)\|_{L_1}\leq M_2 \left(\sum_j D(c_j,\lambda_i)\right).$$ 
Since $\nu_{\lambda_i}$  converges $*$-weakly to $0$ then by Lemma \ref{lm.weakaccum}, we have $\displaystyle{\lim_{i\rightarrow \infty}}\| (1-\lambda)A_{R(c)}(z,\lambda_i)\|_{L_1}=0$, thus for every $j$ the limit $\displaystyle{\lim_{i\rightarrow \infty}} D(c_j,\lambda_i)=0$ which finishes the proof. 
 \end{proof}

We have the following. 
 
 \begin{proposition}\label{the.tec.principal}
 Let $R$ be as in Lemma \ref{lm.measures}. If the measures $\nu_{\lambda_i}$ converge $*$-weakly to $0$ for a suitable sequence of reals $\lambda_i \rightarrow 1$ and a critical point $c\in J(R)$  then $$\lim_{i\rightarrow \infty}\frac{1}{1-\lambda_i} \int 
 \gamma_a(z)d\nu_{\lambda_i}(a)=0$$ for almost every $z\in D(R)$. Moreover, on the Fatou set the limit above is uniform
on compact subsets outside the postcritical set $P(R)$.

 \end{proposition}
 
 \begin{proof}
First let us show that for any $\phi\in L_1(\C)$ the series  $\sum R_*^{n}(\phi)$ is finite 
and converges absolutely almost everywhere on $D(R)$. It is enough to show that $\sum |R_*^{n}(\phi)|$ is 
integrable on any wandering set $W$. Direct computations show
$$\int_W |R_*^{n}(\phi)(z)||dz|^2\leq \int_{R^{-n}(W)}|\phi(z)||dz|^2,$$  which yields 
$$\int_{W}\sum_{n=0}^\infty |R_*^{n}(\phi)(z)|\leq \sum_{n=0}^\infty \int_{R^{-n}(W)}|\phi(z)||dz|^2\leq \int_\C |
\phi(z)||dz|^2.$$
In particular, $R_*^{n}(\phi)(z)$ converges to $0$ almost everywhere on $D(R).$ 

Second, if $z_0\in F(R) \setminus P(R)$, we can find a disk $D_0\subset D(R)$ 
centered at $z_0$. Now suppose that  $\phi$ is holomorphic on $\C\setminus P(R)$. 
Since we have $$\int_{D_0}|R_*^{n}(\phi)(z)|dz|^2\leq \int_{\C}|\phi(z)||dz|^2,$$ by the 
mean value theorem $\{R_*^{n}(\phi)\}$ forms a normal family of holomorphic functions on $D_0.$ By the discussion above $R_*^{n}(\phi)$ and $\sum R_*^{n}(\phi)(z)$ converge to 
their respective limits uniformly on compact subsets of $D_0.$  Then by the Abel 
theorem we get 
$$\lim_{\lambda_i \rightarrow 1}\sum \lambda_i^n R_*^{n}(\phi)(z)=\sum R_*^{n}
(\phi)(z)$$ almost everywhere on  $D(R)$ and uniformly 
on compact subsets of $F(R)\setminus P(R).$

 To finish the proof, we take  $\lambda_i\to 1$ and apply Lemma \ref{lm.measures} and the discussion above  to the
 Equation \eqref{eq.posrad} of Corollary \ref{cor.posrad}  to get 
 $$\lim_{i\rightarrow \infty} A_{R(c)}(z,\lambda_i)=\lim_{i\rightarrow \infty}\frac{1}
 {1-\lambda_i} \int_\C \gamma_a(z)d\nu_{\lambda_i}(a)=0$$ almost everywhere on the dissipative set $D(R)$ and uniformly on compacts subsets of $F(R)\setminus P(R).$\end{proof}
Now we are ready to prove the Abel version of Theorem \ref{th.NorlundTop}.

\begin{theorem}\label{TheoremB}
 Assume that for a critical point $c\in J(R)$ the sequence  measures $\{\nu_{\lambda_i}\}$ converges $*$-weakly to $0$ for a suitable sequence $\lambda_i<1$ converging to $1.$ Then $R$ is an unstable map whenever $c$ and the individual  postcritical set 
 $P_c(R)$ are separated by the Fatou set.
\end{theorem}

 \begin{proof}
By contradiction, assume that $R$ is structurally stable. Then by Lemma \ref{lm.measures} we get $$\lim_{\lambda_i\to1}A_{R(c)}(c,
\lambda_i)=R''(c)\neq 0.$$
On the other hand, by Proposition \ref{the.tec.principal} we have that $$A_{R(c)}(z,\lambda_i)=\frac{1}{1-\lambda_i}\int \gamma_a(z)d\nu_{\lambda_i}(a)$$ converges to $0$ uniformly on compact 
subsets of $F(R)\setminus P(R).$ By assumption we can select a Jordan curve $\gamma \subset 
F(R)\setminus P(R)$ separating $c$ and $P_c(R)$. Since $A_{R(c)}(z, \lambda_i)$ is holomorphic for $z$ in the interior of $\gamma$,  by Cauchy's theorem we have
$$ \lim_{\lambda_i\rightarrow 1} A_{R(c)}(c,\lambda_i)=\frac{1}{2\pi i}\int_{\gamma} \frac{A_{R(c)}(z,
\lambda_i)}{z-c}dz=0,$$ a contradiction. \end{proof}
 
\begin{proof}[Proof of Theorem \ref{bounded}]
The first part of the Theorem is a consequence of Lemma \ref{lem.36} and Theorem \ref{TheoremA}.

For the last part, we proceed by contradiction. Assume  $R$ is structurally stable. Since the measures $\nu_\lambda$
form a uniformly bounded family of measures  for $0\leq \lambda <1$ then by Theorem \ref{th.NorlundTop}, every $*$-weak limit of $\nu_\lambda$ for $\lambda\rightarrow 1$ is a non-zero Abel measure. 

(2)  Since
$P_c(R)$ has measure zero, then we have a contradiction with Theorem \ref{TheoremA}. 

(3) Assume that $\sigma(c)$ is convergent.  Let $\tau\neq 0$ be an Abel measure which is not an $M$-measure. Then 
$$f(z)=\int_\C \gamma_a(z)d\tau(a)$$ is a non-zero integrable function on $\C$ supported on 
$P_c(R)$ satisfying $R_*(f)=f.$ By Lemma 11 in \cite{MakRuelle} there exists an 
invariant Beltrami differential $\mu$ with $\mu(z)=\frac{|f(z)|}{f(z)}$ almost 
everywhere on the support of $f.$
Computations give rise to \begin{align*}0&\neq\int_\C |f||dz|^2=\int_\C \mu(z)f(z)|dz|^2\\
&=\lim_{\lambda_i\rightarrow 1}(1-\lambda_i)\int_\C \mu(z) A_{R(c)}(z,\lambda_i)|dz|^2\\
&=\lim_{\lambda_i\rightarrow 1}(1-\lambda_i)\left[\left(\frac{1}{R''(c)}A_{R(c)}(c,
\lambda_i)-1\right) \int_\C \mu(z)B_{R(c)}(z,\lambda_i)|dz|^2\right.+\\
&\left.\hskip 2cm+ 
\sum_{\tilde{c}\in Crit(R)\setminus \{c\}} \frac{1}{R''(\tilde{c})}A_{R(c)}(\tilde{c},
\lambda_i)\int_\C \mu(z)B_{R(\tilde{c})}(z,\lambda)|dz|^2\right].\end{align*}
By the invariance of $\mu$, this reduces to 
\begin{equation}\label{cuatro}\tag{3}0\neq F_\mu(v)\lim_{\lambda_i\rightarrow 1}\left[\frac{1}{R''(c)}A_{R(c)}(c,
\lambda_i)-1\right]+\sum_{\tilde{c}\in Crit(R)\setminus \{c\}}\frac{F_\mu(\tilde{v})}
{R''(\tilde{c})}\lim_{\lambda_i\rightarrow 1} A_{R(c)}(\tilde{c},\lambda_i).\end{equation}

Now, as $R$ is stable and  since $\sigma(c)$ is convergent,   
then Corollary by \ref{cor.motivation} we have that  $\sigma(c)$ converges to $0$. Let 
$s_n(\bar{c})$ be the partial sums of the formal series $A_{R(c)}(\bar{c})$, here $\bar{c}$ is any critical point. Then the assumptions, Lemma \ref{lm.struc.invariant}, and the formula after Lemma \ref{lem.ratcrit} together with the fact that $\sigma(c)$ converges to $0$
yield
$$\lim_{n\rightarrow \infty} 
s_n(\bar{c})= 
\begin{cases}
  \phantom{-}0  & \text{if}\ \bar{c}\neq c \\
  R''(c)           & \text{if}\ \bar{c}=c.
\end{cases}
$$ Abel's theorem then gives $$\lim_{\lambda\rightarrow 1}A_{R(c)}(\bar{c},
\lambda)=(1-\lambda)\sum_n s_n(\bar{c})\lambda^n= 
\begin{cases}
  \phantom{-}0  & \text{if}\ \bar{c}\neq c \\
  R''(c)           & \text{if}\ \bar{c}=c.
\end{cases}
$$ After replacing these values in (\ref{cuatro}) we achieve the desired contradiction.  \end{proof}

Note that the condition that $P_c(R)$ has measure zero is used to guarantee that the measure $\tau$ is an $M$-measure.  As mentioned in the introduction, this condition can be dropped for maps with a completely invariant Fatou domain.

With small modifications the theorems above can be extended to the case of entire or 
meromorphic functions with finitely many critical  and asymptotic values. 
\subsection{The Voronoi case}\label{Voronoicase}
Now we prove of Theorem \ref{th.Norlundregular} which extends the ideas of the Abel versions of
Theorem \ref{TheoremA} and Theorem \ref{TheoremB} to the Voronoi case.  

\begin{proof}[Proof of Theorem \ref{th.Norlundregular}] 
As $|\sigma_n(c)|$ is N\"orlund regular with respect to the matrix $N=\{\frac{q_{n-
m}}{Q_n}\}$, it has radius of convergence $r>0.$
 Hence $$e_\lambda=(1-\lambda)\sum \lambda^n\sigma_n(c)\delta_{R^n(c)}$$ is a finite measure for $|\lambda|<r$. By Part (3) 
 of Lemma \ref{lemma.Norlund}  we have $$q(\lambda)\cdot e_\lambda=\nu_\lambda.$$
 Thus $e_\lambda$ can be extended to the open unit disk as a meromorphic family of measures which is holomorphic on a neighborhood of $[0,1)$. Let $\displaystyle{E_\lambda=\frac{\nu_\lambda}{q(\lambda)}}$ be the induced extension. In this way the function 
 $$A_{v}(z,\lambda)=\frac{1}{1-\lambda}\int_\C \gamma_a(z)de_\lambda(a)$$
 extends to
\begin{align*} E_{v}(z,\lambda)&=\frac{1}{q(\lambda)(1-\lambda)}\int_\C \gamma_a(z)d\nu_\lambda(a)
 \\&=\frac{1}{(1-\lambda)}\int_\C \gamma_a(z)dE_\lambda
 (a).\end{align*}
 In other words, for $z$ outside $P_c(R)\cup \{0,1,\infty\}$ the sequence 
$\{\sigma_n(c)\gamma_{R^n(v)}(z)\}$ is N\"orlund regular with respect to $N$. 

Like in the proofs of Lemma \ref{lm.struc.invariant} and Corollary \ref{cor.posrad}, we  extend  
$A_{v}(\tilde{c},\lambda)$ to a meromorphic function $E_{v}(\tilde{c},\lambda)$ for every critical point $\tilde{c}$  so that   
\begin{equation}\label{eq.theo14} B_{v}(z,\lambda)=E_{v}(z,\lambda)+\lambda \sum_{c_i\in Crit(R)} \frac{1}{R''(c_i)}E_{v}(c_i,
\lambda)B_{v_i}(z,\lambda)\tag{***}
 \end{equation}
 holds on a neighborhood of $[0,1).$

Under the assumptions, the arguments of  Proposition \ref{prop.strucintegrab} apply. \end{proof}

\begin{proof}[Proof of Theorem \ref{th.NorlundTop}]
We proceed as in Theorem \ref{TheoremB}, but we only need to apply 
Lemma \ref{lm.struc.invariant} and Proposition \ref{the.tec.principal} to Equation 
\eqref{eq.theo14}. If $R$ is structurally stable we get a contradiction to 
the assumptions proceeding as in Theorem \ref{TheoremB}.\end{proof}

So far we have consider the N\"orlund-Voronoi method to produce  finite non-zero measures. A general averaging mechanism can be hinted to establish the Fatou conjecture for a wider class of sequences.

Let $\Theta$ be the space of all complex sequences.
An infinity matrix $M$ is \textit{regular} if when restricted to $\mathcal{C}$, the space of converging sequences, it defines a continuous operator  such 
that if $\{t_n\}=M\{a_n\}$ then $\lim a_n=\lim t_n$. A basic fact here is Agnew's theorem which states that given $x,y\in \Theta$ if either 
\begin{enumerate}
 \item $x\in \ell_\infty \setminus \mathcal{C}$ and $y\in \ell_\infty$ or
 \item $x\in \Theta\setminus \ell_\infty$ and $y\in \Theta,$
 
 \end{enumerate}then there is regular matrix $M$ with $Mx=y$ (see Theorem 2.6.4 in
\cite{Boos}). The matrix $M$ is not unique. Moreover, the space of all regular matrices sending $x$ to $y$ is infinitely dimensional. 

Let us recall that $\ell_1$ acts on $\Theta$ 
 by convolutions. Take an infinite matrix $M$ subject to the following two conditions.
 
\begin{itemize}
\item[i)] If $c\in J(R)$ then $M$ transforms the measures 
$\rho_k=\sum_{n=0}^k\delta_{R^n(v)}\sigma_n(c)$  onto 
 measures $t_k$ with $\sup_k \|t_k\|<\infty$ and there is a non-zero $\alpha$ which is 
 a $*$-weak accumulation point of the sequence $\{t_k\}.$  
(Agnew's theorem guarantees the existence of such a matrix for non convergent 
complex sequences).

 \item[ii)]  The matrix $M$ commutes with the action of $\ell_1$ by convolutions. 
 \end{itemize}
When this happens, using the formal  Equation \eqref{eq.Poincare} we can show the 
instability of the corresponding rational map whenever $\alpha$ is an $M$-measure. 
Although the algebra of linear operators on $\Theta$ commuting with the 
action of $\ell_1$, by convolutions,  is non-separable and the space of 
regular matrices mapping $x$ to $y$ is infinitely dimensional, the authors were not able to find a matrix satisfying the conditions (i) and (ii) above.  On the other hand, 
most of the summation methods discussed, for instance, in the book by J. Boos \cite{Boos}
satisfy both conditions. In this paper, we used one of the most general summation methods.

\bibliographystyle{amsplain}
\bibliography{workbib}
\end{document}